\documentclass{amsart}

\usepackage{amsthm, amsmath, amsfonts, hyperref, verbatim}
\theoremstyle{plain}
\newtheorem{thm}{Theorem}
\newtheorem{lem}[thm]{Lemma}
\newtheorem{cor}[thm]{Corollary}

\theoremstyle{definition}
\newtheorem{exl}[thm]{Example}

\newcommand{\Z}{\mathbb{Z}}

\DeclareMathOperator{\Rem}{Rem}
\DeclareMathOperator{\Fix}{Fix}

\newcommand{\Growth}{\mathop{\mathrm{Growth}}}

\begin{document}

\bibliographystyle{hplain}

\title{Dynamics of random selfmaps of surfaces with boundary}
\author{Seung Won Kim}
\address{Dept. of Mathematics, Kyungsung University, Busan 608-736, Republic of Korea}
\email{kimsw@ks.ac.kr} 
\author{P. Christopher Staecker}
\address{Dept. of Mathematics and Computer Science, Fairfield University, Fairfield CT, 06824, USA}
\email{cstaecker@fairfield.edu}

\begin{abstract}
We use Wagner's algorithm to estimate the number of periodic points of certain selfmaps on compact surfaces with boundary. When counting according to homotopy classes, we can use the asymptotic density to measure the size of sets of selfmaps. 
In this sense, we show that ``almost all'' such selfmaps have periodic points of every period, and that in fact the number of periodic points of period $n$ grows exponentially in $n$. We further discuss this exponential growth rate and the topological and fundamental-group entropies of these maps.

Since our approach is via the Nielsen number, which is homotopy and homotopy-type invariant, our results hold for selfmaps of any space which has the homotopy type of a compact surface with boundary.
\end{abstract}
\maketitle

\section{Preliminaries}
Let $X$ be a space with the homotopy type of a compact connected surface with boundary, and let $f:X \to X$ be a selfmap. We will study the fixed and periodic point theory of $f$ using the Nielsen number $N(f)$. This $N(f)$ is a natural number, possibly zero, which is a homotopy invariant lower bound for the size of the fixed point set $\Fix(f)$. 

Let $G = \pi_1(X)$, which will be a finitely generated free group. Throughout the paper we will let $m$ be the rank of $G$, fix a particular free basis $\{ a_1, \dots, a_m \}$, and consider elements of $G$ as reduced words in the letters $a_i^{\pm 1}$. The Nielsen number of $f$ is a homotopy invariant which can be computed directly from the induced homomorphism $f_\#:\pi_1(X) \to \pi_1(X)$. For an endomorphism $\phi:G \to G$, we let $N(\phi)$ be the Nielsen number of some map $f:X \to X$ with $f_\# = \phi$. 

Our main tool for computing the Nielsen number will be Wagner's algorithm of \cite{wagn99}. There exists a well-developed Nielsen theory of periodic points for selfmaps on manifolds. Wagner's algorithm has recently been adapted by Hart, Heath, and Keppelmann in \cite{hhk1} and \cite{hhk2} for  Nielsen periodic point theory on surfaces with boundary. We will use some of these sophisticated methods in our approach, but generally the ordinary Nielsen number $N(f^n)$ will suffice.

Wagner's algorithm applies to maps which satisfy Wagner's remnant condition. 
For an endomorphism $\phi:G \to G$, we say that $\phi$ \emph{has remnant} when, for each $i$, there
is a nontrivial subword of $\phi(a_i)$ which does not cancel in any product
of the form
\[ \phi(a_j)^{\pm 1} \phi(a_i) \phi(a_k)^{\pm 1} \]
except for when $j$ or $k$ equals $i$ and the exponent is $-1$.
Each such maximal
noncanceling subword is called the remnant of $a_i$, denoted $\Rem_\phi
a_i$.

The statement that $\phi$ has remnant is similar to
the statement that the set $\{\phi(a_i)\}$ is Nielsen reduced (see e.g.\ \cite{ls77}).  

We review the basics of Wagner's algorithm: Given an endomorphism $\phi:G \to G$, we first build the set of \emph{Wagner tails}, which are pairs of elements of $G$. This set is constructed as follows: for any $\phi$, the pair $(1,1)$ is a Wagner tail. Also, for each occurrence of the letter $a_i^\epsilon$ (for $\epsilon \in \{+1, -1\}$) in $\phi(a_i)$, write a reduced product $\phi(a_i) = va_i^{\epsilon}\bar v$. Then $(w,\bar w)$ is a Wagner tail, where 
\[ w = \begin{cases} 
v &\text{ if } \epsilon = 1 \\
va_i^{-1} &\text{ if } \epsilon = -1 \end{cases} \quad
\bar w = \begin{cases}
\bar v^{-1} &\text{ if } \epsilon = 1 \\
\bar v^{-1}a_i &\text{ if } \epsilon = -1 \end{cases}
\]
In this case we say that the Wagner tail $(w,\bar w)$ \emph{arises from an occurence of $a_i^\epsilon$ in $\phi(a_i)$}. 
The \emph{fixed point index} of such a Wagner tail is equal to $-\epsilon$, the opposite of the exponent on the letter from which it arises.

Let $(w_1, \bar w_1)$ and $(w_2, \bar w_2)$ be Wagner tails. Then we say that $(w_1, \bar w_1)$ and $(w_2, \bar w_2)$ are \emph{directly related} when $\{w_1, \bar w_1\} \cap \{w_2, \bar w_2\} \neq \emptyset$. Two Wagner tails $(w, \bar w)$ and $(w', \bar w')$ are \emph{indirectly related} when there is a sequence of Wagner tails $(w,\bar w) = (w_0, \bar w_0), (w_1, \bar w_1), \dots, (w_k, \bar w_k) = (w', \bar w')$ such that $(w_i, \bar w_i)$ is directly related to $(w_{i+1}, \bar w_{i+1})$ for each $i$. This indirect relation is an equivalence relation, and the equivalence classes of Wagner tails are called the \emph{fixed point classes}. Wagner showed that, when $\phi$ has remnant, the number of fixed point classes having nonzero fixed point index sum is equal to $N(\phi)$. 

In this paper we will not compute $N(\phi)$ exactly, but we obtain a lower bound based on the following easy observation: $N(\phi)$ is greater than or equal to the number of Wagner tails which are not directly related to any others.

The set of homomorphisms with remnant (for which Wagner's algorithm applies) is very large. In fact, ``most'' endomorphisms have remnant, when measured as follows according to asymptotic density.

Let $G_p$ be the subset
of all words of $G$ of length at most $p$. The \emph{asymptotic density} (or simply
\emph{density}) of a subset $S\subset G$ is defined as 
\[ D(S) = \lim_{p \to \infty} \frac{|S \cap G_p|}{|G_p|}, \]
where $|\cdot|$ denotes the cardinality. The set $S$ is said to be
\emph{generic} if $D(S) = 1$. The density $D(S)$ can be thought of as the probability that a random element of $G$ is in the set $S$.

Similarly, if $S \subset G^k$ is a set of $k$-tuples of
elements of $G$, the asymptotic density of $S$ is defined as
\[ D(S) = \lim_{p \to \infty} \frac{|S \cap (G_p)^k|}{|(G_p)^k|}, \]
and $S$ is called \emph{generic} if $D(S) = 1$. 

An endomorphism $\phi:G \to G$ is equivalent combinatorially to a $m$-tuple of elements of $G$
(the $m$ elements are the words $\phi(a_1), \dots, \phi(a_m)$). Thus
the asymptotic density of a set of 
homomorphisms can be defined in the same sense as above, viewing the set of
homomorphisms as a collection of $m$-tuples. 

A theorem of Robert F.\ Brown in \cite{wagn99} established that
the set of endomorphisms with remnant is generic. Lemma 9 of \cite{stae11} (based on the work in \cite{ao96})
shows that in fact $R_k$ is generic for any $k$, where $R_k$ is the set of endomorphisms $\phi$ with $|\Rem_\phi a_i| \ge k$ for all $k$. (Throughout, for a word $w\in G$, we write $|w|$ for the word length of $w$ as a reduced word.)

In settings (such as Nielsen theory) where homotopy-invariant properties of continuous mappings are studied, there is a natural formulation of genericity for sets of continuous mappings. For a space $X$ with $G = \pi_1(X)$ finitely generated free, and a set $A$ of continuous selfmaps on $X$, define $A_\# =  \{ f_{\#} \mid [f] \subset A \}$, where $[f]$ is the homotopy class of $f$ and $f_\#$ is the induced homomorphism on $G$.
Then we say that $A$ is \emph{homotopy-generic} when $A_\#$ is generic. In this case we can informally say that, when considered according to homotopy classes, ``almost all'' selfmaps of $X$ are in $A$.

In this paper we will show that, in the above sense, ``almost all'' selfmaps of a bouquet of circles (or any space of this homotopy type) have periodic points of every period $n$, and in fact that the number of such points increases exponentially in $n$. We additionally derive a lower bound for the exponential growth rate of this quantity.

In the case where $m$, the rank of the group $G$, is 1, these results follow easily. In this case $G$ is isomorphic to the integers $\Z$, and each homomorphism is multiplication by some $d\in \Z$. Denote this homomorphism by $\phi_d:\Z \to \Z$. It is a classical result in Nielsen theory (see \cite{jian83}) that $N(\phi_d) = |1-d|$. Furthermore it is easy to see that $\phi_d^n = \phi_{d^n}$. Thus we will have $N(\phi^n) = |1-d^n|$, and this quantity grows exponentially whenever $|d| > 1$. 

The paper is organized as follows: In Section \ref{Wsection} we define a generic set of homomorphisms $S_l$ and show that all maps $f$ with $f_\#\in S_l$ have many periodic points of all periods. In Section \ref{mpersection} we show that these maps in fact have many periodic points of all \emph{minimal} periods. In Section \ref{othersection} we discuss some other dynamical invariants of these maps including the topological entropy and the asymptotic Nielsen number. The proof that $S_l$ is generic requires a detailed counting argument which is given in Section \ref{Dsection}.

\section{Periodic points of maps in $S_l$}\label{Wsection}
In this section we define a set of endomorphisms $S_l$ and show that for each map $\phi \in S_l$, the sequence $\{N(\phi^n)\}$ is nonzero and grows exponentially. Since $N(f^n) \le \#\Fix(f^n)$, this will imply that the sequence of numbers of periodic points $\{\#\Fix(f^n)\}$ is also nonzero and grows exponentially whenever $f_\# \in S_l$. 

Let $(w, \bar w)$ be a Wagner tail which arises from an occurrence of $a_i^{\pm 1}$ inside the word $\phi(a_i)$. If $\phi$ has remnant and this occurrence of $a_i^{\pm 1}$ is an inside letter (not the first or last) of the subword $\Rem_\phi a_i$, then 
we say that $(w,\bar w)$ \emph{occurs inside the remnant}.
For any endomorphism $\phi$, let $W(\phi)$ be the set of Wagner tails which occur inside the remnant.

\begin{lem}
For any homomorphism $\phi$ and any $(w,\bar w) \in W(\phi)$, the Wagner tail $(w,\bar w)$ is not directly related to any other Wagner tail of $\phi$. 
\end{lem}
\begin{proof}
The proof follows easily from the definition of the Wagner tails, though we must check several cases. Let $(w_1,\bar w_1) \in W(\phi)$ and let $(w_2,\bar w_2)$ be some other Wagner tail (perhaps not in $W(\phi)$), and we will argue that $\{w_1, \bar w_1\} \cap \{ w_2, \bar w_2\}$ is empty.

The easiest case is when $w_1$ arises from an occurrence of $a_i$ in $\phi(a_i)$ and $w_2$ arises from an occurrence of $a_j$ in $\phi(a_j)$. Then we have reduced products:
\[ \phi(a_i) = w_1 a_i \bar w_1^{-1}, \quad \phi(a_j) = w_2 a_j \bar w_2^{-1}. \]
Since $(w_1,\bar w_1)\in W(\phi)$, each word $w_1$ and $\bar w_1^{-1}$ contains at least one letter of the remnant subword $\Rem_\phi a_i$. Because of this, it is impossible for any of $w_1, \bar w_1, w_2, \bar w_2$ to be equal, since this would imply that part of the remnant subword could cancel in some product $\phi(a_i)^{\pm 1} \phi(a_j)^{\pm 1}$. Thus $\{w_1, \bar w_1\} \cap \{ w_2, \bar w_2\}$ is empty.

Similarly we must check three other cases: $w_1$ arises from an occurrence of $a_i^{-1}$ in $\phi(a_i)$ and $w_2$ arises from an occurrence of $a_j$ in $\phi(a_j)$, $w_1$ arises from an occurrence of $a_i$ in $\phi(a_i)$ and $w_2$ arises from an occurrence of $a_j^{-1}$ in $\phi(a_j)$,  and $w_1$ arises from an occurrence of $a_i^{-1}$ in $\phi(a_i)$ and $w_2$ arises from an occurrence of $a_j^{-1}$ in $\phi(a_j)$. All of these cases are analogous to the above, and we omit the details.
\end{proof}

Any Wagner tail which is not directly related to any other must represent its own fixed point class with nonzero index, and so we obtain:
\begin{lem}\label{ngew}
For any $\phi$, we have $N(\phi) \ge \#W(\phi)$.
\end{lem} 

For any letter $a_i$ and reduced word $w$, let $\Phi_{a_i}(w)$ be the number of times that $a_i$ or $a_i^{-1}$ appear in $w$.
Our main result for this section is that $\{\#W(\phi^n)\}$ (and thus also $\{N(\phi^n)\}$) grows exponentially when $\phi$ is a member of the following set:
For some $l> 0$, let $S_l$ be the set of maps $\phi$ such that for any generators $a_i, a_j$ of $G$, we have $\Phi_{a_i} (\Rem_\phi a_j) \ge l$. 

We will require a lemma concerning the relationship between $\Rem_{\phi^{n-1}}a_i$ and $\Rem_{\phi^n} a_i$. 

\begin{lem}\label{iteratesrem}
Let $\phi$ have remnant with $\Phi_{a_i}(\Rem_\phi a_j) \ge l$ for some $i$ and $j$. Then $(\Rem_{\phi^{n-1}} a_i)^{\pm 1}$ occurs $l$ disjoint times (possibly with different exponents) as a subword of $\Rem_{\phi^n} a_j$. 
\end{lem}
\begin{proof}
Our hypothesis implies that $a_i^{\pm 1}$ appears at least $l$ times in $\Rem_\phi a_j$.
Then there are words $u_k$ where 
\begin{equation}\label{ajletters}
\phi(a_j) = u_0 a_i^{\epsilon_1} u_1 a_i^{\epsilon_2} \dots a_i^{\epsilon_l} u_l 
\end{equation}
 is a reduced product where each $\epsilon_k \in \{+1, -1\}$, each $a_i^{\epsilon_k}$ above appears inside the remnant, and possibly some $u_k$ are trivial (but in that case $\epsilon_k = \epsilon_{k+1}$).

To examine $\Rem_{\phi^n}a_j$, we look at products of the form $\phi^n(b)\phi^n(a_j)\phi^n(c)$ where $b$ and $c$ are letters other than $a_j^{-1}$. 
Let $x = \phi(b)$ and $y=\phi(c)$, and we have
\begin{align*} 
\phi^n(b)\phi^n(a_j)\phi^n(c) 
&= \phi^{n-1}(x) \phi^{n-1}(\phi (a_j)) \phi^{n-1}(y) \\
&= \phi^{n-1}(x) 
[ \phi^{n-1}(u_0) \phi^{n-1}(a_i)^{\epsilon_1} \dots \phi^{n-1}(a_i)^{\epsilon_l} \phi^{n-1}(u_l)]  \phi^{n-1}(y).
\end{align*}
Inside the brackets above, each term contains a remnant which will not cancel with other terms inside the bracket. These remnants may cancel, however, when we take the products with $\phi^{n-1}(x)$ and $\phi^{n-1}(y)$.

We will argue that in fact none of the remnants of $\phi^{n-1}(a_i)^{\epsilon_k}$ inside the brackets above will cancel.
Let us assume for the sake of contradiction that, for example, $(\Rem_{\phi^{n-1}} a_i)^{\epsilon_1}$ cancels inside $\phi^{n-1}(a_i)^{\epsilon_1}$ above due to the product with $\phi^{n-1}(x)$. 

Lemma 3.6 of \cite{wagn99} states that if $\phi$ has remnant, then so does $\phi^k$ for any $k$. Thus $\phi^{n-1}$ has remnant, and so
the cancellation assumed above is only possible when there is some word $w$ such that $x$ has reduced form $x= w a_i^{-\epsilon_1}u_0^{-1}$.

Recall, however, that $x=\phi(b)$ for some letter $b\neq a_j^{-1}$. If we have $x= w a_i^{-\epsilon_1}u_0^{-1}$, there will be cancellation in the product $\phi(b)\phi(a_j)$. In particular $\Rem_{\phi}a_j$ must not include the occurance of $a_i^{\epsilon_1}$ from \eqref{ajletters}, which is a contradiction.
\end{proof}

An easy consequence of the above will also be useful:
\begin{lem}
If $\phi\in S_l$, then $\phi^n\in S_{l}$ for all $n>0$.
\end{lem}
\begin{proof}
The proof is by induction on $n$. The case $n=1$ is assumed. For the inductive case, we must show that $\Phi_{a_i}(\Rem_{\phi^n}a_j) \ge l$ for any $i$ and $j$. Since $\phi\in S_l$ we have $\Phi_{a_j}(\Rem_{\phi} a_j) \ge l$, and so Lemma \ref{iteratesrem} says that $(\Rem_{\phi^{n-1}}(a_j))^{\pm 1}$ occurs as a subword of $\Rem_{\phi^n}a_j$. Then we have
\[ \Phi_{a_i}(\Rem_{\phi^n}a_j) \ge \Phi_{a_i} (\Rem_{\phi^{n-1}} a_j) \ge l, \] 
where the last inequality is the induction hypothesis. 
\end{proof}

Our main result for this section follows easily from the following more technical fact:
\begin{lem}\label{philemma}
For any $\phi\in S_l$ and any generators $a_i$ and $a_j$, we have
\[ \Phi_{a_i}(\Rem_{\phi^n}a_j) \ge l^{n} m^{n-1}. \]
\end{lem}
\begin{proof}
The proof is by induction. The case $n=1$ is simply that $\Phi_{a_i}(\Rem_\phi a_j) \ge l$, which is implied by our assumption that $\phi\in S_l$.

For the inductive case, since $\phi\in S_l$ we know that $\Phi_{a_k}(\Rem_\phi a_j) \ge l$ for each $k$ and $j$. Then by Lemma \ref{iteratesrem} we have that $(\Rem_{\phi^{n-1}}a_k)^{\pm 1}$ appears $l$ disjoint times in $\Rem_{\phi^n} a_j$ for each $k$. Thus we have
\begin{align*}
\Phi_{a_i} (\Rem_{\phi^n} a_j) &\ge \sum_{k=1}^m l \cdot \Phi_{a_i}(\Rem_{\phi^{n-1}} a_k) \ge \sum_{k=1}^m l \cdot l^{n-1} m^{n-2} = l^nm^{n-1}
\end{align*}
where the second inequality is the induction hypothesis.
\end{proof}

Now we prove our main result for this section.
\begin{thm}\label{wbound}
For any $\phi\in S_l$, we have 
\[\# W(\phi^n) \ge l^{n}m^n - 2m. \]
\end{thm}
\begin{proof}
Recall that $\#W(\phi^n)$ is the number of Wagner tails of $\phi^n$ which occur inside the remnant. This will equal the sum over each generator $a_i$ of the number of times that $a_i^{\pm 1}$ appears inside (not as the first or last letter) of $\Rem_{\phi^n} a_i$. Subtracting to account for possible occurrences as the first or last letter, we have
\[
\#W(\phi^n) \ge \sum_{i=1}^m \Phi_{a_i} (\Rem_{\phi^n} a_i) - 2. 
\]

By Lemma \ref{philemma} we obtain
\[ \#W(\phi^n) \ge \sum_{i=1}^m l^nm^{n-1} - 2 = l^nm^n - 2m \]
as desired.
\end{proof}

Since, as we show in Section \ref{Dsection}, the set $S_l$ is generic for any $l$, the above implies that there are generic sets on which $\#W(\phi^n)$ is arbitrarily high for any $n$. 
This will imply strong statements about homotopy-generic behavior of Nielsen numbers and periodic points of continuous maps. Since $\#\Fix(f) \ge N(f) \ge \#W(f_\#)$, we obtain:
\begin{cor}\label{periodiccor}
For any real number $r$, the set of selfmaps with
$ \#\Fix(f^n) \ge r $
for all $n$ is homotopy-generic.
\end{cor}

\section{Periodic points with minimal periods}\label{mpersection}
The material above concerns periodic points of $f$, that is, fixed points of $f^n$. Using a slightly more sophisticated argument we can estimate the number of periodic points of $f$ with minimal period $n$, that is, fixed points of $f^n$ which are not fixed points of $f^m$ for any $m<n$. The basic argument is illustrated in the following result, which will be strengthened later. 

\begin{thm}\label{minperiod}   
For $f_\# = \phi\in S_3$, the map $f$ has a periodic point of minimal period $n$ for each $n$.
\end{thm}

In this section we will follow the notation in Section 4 of \cite{hhk1}. When $\phi$ is the induced homomorphism of some continuous mapping $f$ given in a linearized ``standard form'' (any map can be changed by homotopy to its standard form), each fixed point of $f$ is given an ``address'' of the form $r_1 \dots r_n$. The definition of the addresses is a bit technical, but it will be illustrated in our proof. Each $r_i$ appearing in the address refers to the location of a letter in the list $(\phi(a_1), \dots, \phi(a_m))$. For example, for the homomorphism given by
\begin{equation}
\phi(a) = a^3b^3\underline{a}, \quad \phi(b) = b^2a\underline{b}a, 
\end{equation}
the underlined letter $a$ has location $7$ (it is the seventh letter appearing above), and the underlined $b$ has location $11$. More details are given in \cite{hhk1}.

\begin{proof}[Proof of Theorem \ref{minperiod}]
We will prove that there exists a fixed point $x$ with address $r_1r_2\cdots r_n$ of $\phi^n$ such that there is no divisor $d$ of $n$ ($d<n$) such that $r_i=r_j$ whenever $i\equiv j \mod d$. Then by Lemma 4.9 in~\cite{hhk1}, the minimal period of $x$ is $n$. 

For each $i=1, 2, \cdots, m$, let 
$$
s_i=\sum_{j=1}^i |f(a_j)|.
$$

Assume that $n\geq 2$ (the statement of the theorem for $n=1$ is obvious), and choose some $k\in \{1,\dots,m\}$.

Since for any generator $a_j$ of $G$, we have $\Phi_{a_j}(\Rem_\phi a_k)\geq 3$, we can choose a letter $a^{\pm 1}_{j_1}$ in the inside of $\Rem_\phi a_1$ with ${j_1} \neq k$. Choose one such letter $a^{\pm 1}_{j_1}$ 
and let $r_1$ be the location corresponding to the chosen letter. Then we have $s_{k-1} < r_1\leq s_k$ where $s_0=0$.
(For example when $m=3$ and $k=1$ and $\phi(a_1)=a_2a_1a_3^{-1}a_1$, if $a_{j_1}=a_2$ then $r_1=1$ and if $a_{j_1}^{-1}=a_3^{-1}$ then $r_1=3$.)

Since $\Phi_{a_j}(\Rem_\phi a_{j_1})\geq 3$ for any generator $a_j$, we can choose a letter $a^{\pm 1}_{j_2}$ in    
$\phi (a_{j_1})$ with $j_2\neq k$. Let $r_2$ be the location corresponding to the chosen letter $a^{\pm 1}_{j_2}$. 
Then we have $s_{j_1-1}< r_2\leq s_{j_1}$, and since $j_1\neq k$, we have $r_1\neq r_2$.

For each $i=3,4,\cdots, n-1$, repeating the same arguments, we can choose a letter $a^{\pm 1}_{j_i}$ in $\phi (a_{j_{i-1}})$ with $j_i\neq k$ and so the corresponding location $r_i$ satisfies $s_{j_{i-1}-1}< r_i\leq s_{j_{i-1}}$. 
Since $j_{i-1}\neq k$, we also have $r_1\neq r_i$.

Since $\Phi_{a_k}(\Rem_\phi a_{j_{n-1}})\geq l$, there exists $a^{\pm 1}_k$ in $\phi (a_{j_{n-1}})$. Choose one such $a^{\pm 1}_k$
and let $r_n$ be the corresponding location. Since $k\neq j_{n-1}$, we have $r_1 \neq r_n$ as well.

By Lemma 4.6 in~\cite{hhk1}, there is a unique fixed point $x$ of $\bar f^n$ (where $\bar f$ is the standard form of $f$) associated with a round trip $(r_1, r_2, \cdots, r_n, r_1)$, i.e, the address of $x$ is $r_1r_2\cdots r_n$. Since for each $i=2, 3, \cdots, n$, we have $r_i\neq r_1$, there is no divisor $d$ of $n$ ($d<n$) such that $r_i=r_j$ whenever $i\equiv j \mod d$. Thus by Lemma 4.9 in~\cite{hhk1}, the minimal period of $x$ is $n$.  

Since $x$ is a fixed point of $\bar f^n$ arising from an occurrence of $a_k$ inside the remnant of $\phi(a_k)$, we can be sure that the periodic point $x$ cannot be removed by homotopy, and so the original map $f$ (being homotopic to $\bar f$) has a periodic point of minimal period $n$. In fact any two distinct periodic points of $\bar f$ of minimal period $n$ constructed in the above way will yield two distinct periodic points of $f$ of minimal period $n$.
\end{proof}

We will illustrate the proof above in an example.

\begin{exl} Let $\phi$ be a map defined by
\begin{align*}
\phi(a)&=abc,\hspace{3in}\\
\phi(b)&=ca^{-1}ba,\\
\phi(c)&=a^{-1}c^{-1}ab.
\end{align*}

Then $\phi\in S_1$, and the unreduced form of $\phi^3$ is
\begin{align*}
\phi^3(a)=&(abc\cdot ca^{-1}ba\cdot a^{-1}c^{-1}ab)\cdot (\underset{6_3}{\underline{a^{-1}}}c^{-1}ab\cdot c^{-1}b^{-1}a^{-1}\cdot ca^{-1}ba\cdot abc)\\
               &\cdot(c^{-1}b^{-1}a^{-1}\cdot b^{-1}\underset{13_3}{\underline{a^{-1}}}ca\cdot abc\cdot ca^{-1}ba),\\
\phi^3(b)=&(c^{-1}\underset{18_3}{\underline{b^{-1}}}a^{-1}\cdot b^{-1}a^{-1}ca\cdot abc\cdot ca^{-1}ba)\cdot (b^{-1}a^{-1}ca\cdot a^{-1}b^{-1}ac^{-1}\cdot 
c^{-1}b^{-1}a^{-1})\\
              &\cdot (a^{-1}c^{-1}ab\cdot c^{-1}b^{-1}a^{-1}\cdot ca^{-1}ba\cdot abc)\cdot (abc\cdot ca^{-1}ba\cdot a^{-1}c^{-1}ab),\\
\phi^3(c)=&(b^{-1}a^{-1}ca\cdot a^{-1}b^{-1}a\underset{33_3}{\underline{c^{-1}}}\cdot c^{-1}b^{-1}a^{-1})\cdot (a^{-1}b^{-1}ac^{-1}\cdot \underset{36_3}{\underline{c^{-1}}}b^{-1}a^{-1}\cdot a^{-1}c^{-1}ab\cdot abc)\\
                &\cdot(abc\cdot ca^{-1}ba\cdot a^{-1}\underset{41_3}{\underline{c^{-1}}}ab)\cdot(a^{-1}c^{-1}ab\cdot c^{-1}b^{-1}a^{-1}\cdot ca^{-1}ba\cdot abc).
\end{align*}
There are 46 fixed points of $\phi^3$ including the base point $0_3$, denoted by $0_3, 1_3, 2_3, \cdots, 45_3$.

We will apply the proof of the above theorem. First, we can choose letters $b$ or $c$ in $\phi(a)$. Take the letter $b$. Then we have $r_1=2$. For $r_2$, we can choose letters $c$ or $b$ in $\phi(b)$. Take the letter $c$ in $\phi(b)$. Then we have $r_2=4$. For $r_3=r_n$, we can choose the letters $a^{-1}$ or $a$ in $\phi(c)$. Choose the letter $a^{-1}$. Then $r_3=8$. The address is
$$
r_1r_2r_3=248
$$
and the fixed point associated with this address is $6_3$. By Lemma 4.9 of \cite{hhk1}, the minimal period of $6_3$ is 3.

In fact in this case we see that that $r_2r_3r_1=482$ and $r_3r_1r_2=824$, and thus the orbit of the fixed point $6_3$ is 
$$
\langle 6_3, 18_3, 33_3\rangle
$$ 
and the minimal period of $6_3$ is 3. 

For (another) example, for the address 
$$ 
r_1r_2r_3=39(10)
$$
the associated periodic point of minimal period 3 is $13_3$ which has orbit
$$
\langle 13_3, 36_3, 41_3\rangle .
$$ 
\end{exl}

The proof of Theorem \ref{minperiod} extends easily to give a lower bound on the number of periodic points of minimal period $n$. Let $P_n(f)$ be the number of periodic points of $f$ of minimal period $n$. We obtain:
\begin{thm}
For $f_\# = \phi \in S_l$ with $l\ge 1$ and for each $n$, we have 
\[ P_n(f) \ge m((m-1)l-2)(m-1)^{n-2}l^{n-1}. \]
\end{thm}
\begin{proof}
The proof of Theorem \ref{minperiod} constructs a point of minimal period $n$ based on several independent choices. The desired bound will be the product of the number of choices at each stage. 

We begin with $k$, for which there are $m$ choices.

For $r_1$ there are at least $m-1$ choices for $j_1$ ($a_{j_1}$ can equal any generator except $a_k$), followed by at least $l$ choices for the location of $a^{\pm 1}_{j_1}$ inside $\Rem_\phi a_k$. Because we must make our selection inside the remnant, we will subtract 2 from the available choices to exclude the first and last letters of the remnant. Thus there are $l(m-1)-2$ choices for $r_1$.

For $r_i$ with $i=2, \dots, n-1$, the situation is similar, though we do not require that $a_{j_i}$ appears inside the remnant. Thus the number of choices for these $r_i$ is at least $l(m-1)$.

For $r_n$ we pick a location of $a_{k}$ in $\phi (a_{j_{n-1}})$, and there are at least $l$ choices for this.

Taking the product, we have 
\[ P_n(f) \ge m \cdot (l(m-1)-2) \cdot (l(m-1))^{n-2} \cdot l, \]
which gives the desired bound.
\end{proof}

Since, as we show in Section \ref{Dsection}, $S_l$ is generic for any $l$ we can improve the result of Corollary \ref{periodiccor}:
\begin{cor}
For any real number $r$, the set of selfmaps with $\# P_n(f) \ge r$ for all $n$ is homotopy-generic.
\end{cor}

\section{Generic behavior of some other dynamical invariants}\label{othersection}

In \cite{jian96}, Jiang defines the \emph{asymptotic Nielsen number} $N^\infty(f)$ as the growth rate of the Nielsen numbers:
\[ N^\infty(f) = \Growth_{n\to \infty} N(f^n) = \max\left\{1, \limsup_{n\to\infty} N(f^n)^{1/n}\right\}. \] 
Jiang shows that for a selfmap on a polyhedron the quantity $N^\infty(f)$ is finite. For selfmaps on the circle ($m=1$) it is easy to see that when $f$ is the degree $d$ map we have $N^\infty(f) = \max\{1,|d|\}$. 

For $m>1$ we may use Theorem \ref{wbound} to estimate $N(f^n)$. When $f_\# \in S_l$ we have $N^\infty(f) \ge \limsup_{n\to\infty} (l^nm^n-2m)^{1/n}$ and thus:
\begin{thm}\label{ninfty}
Whenever $f_\#\in S_l$, we have $N^\infty(f) \ge lm$.
\end{thm}

Jiang also shows that, for polyhedra, the quantity $N^\infty(f)$ is related to the topological entropy $h(f)$ as follows:
\[ h(f) \ge \log N^\infty(f).\] 
Thus we will have $h(f) \ge \log(lm)$ whenever $f_\#\in S_l$.
We will give another proof of this fact based on the \emph{fundamental group entropy} $h_\#(f)$.

We follow the definitions from \cite{kh95}. The fundamental group entropy can be defined for any finitely generated group, but is much simpler for a free group. Let $G = \pi_1(X) = \langle a_1, \dots, a_m\rangle$ and let $\phi:G \to G$ be an endomorphism. Let $L_n(\phi) = \max_{i} |\phi^n(a_i)|$, and we define:
\[ h_\#(f) = \lim_{n\to\infty} \log (L_n(f_\#)^{1/n}). \]
(For non-free groups the lack of a minimal word length makes the definition considerably more complicated.)

Our result is:
\begin{thm}
Let $f_\# \in S_l$. Then we have
\[ h(f) \ge h_\#(f) \ge \log(lm) \]
\end{thm}
\begin{proof}
The inequality $h(f) \ge h_\#(f)$ will always hold in our setting (a proof in \cite{kh95} is stated for compact connected manifolds, but requires only that all sufficiently small loops are contractible). Thus we need only show the second inequality. 

Choose any $i$, and we have 
\begin{align*}
h_\#(f) &\ge \lim_{n\to\infty} \log (|\phi^n(a_i)|^{1/n}) 
\ge \lim_{n\to\infty} \log (\Phi_{a_i}(\Rem_{\phi^n}a_i)^{1/n}) \\
&\ge \lim_{n\to\infty} \log (l^{n} m^{n-1})^{1/n} 
= \lim_{n\to\infty} \log(lm^{1-1/n}) = \log (lm)
\end{align*} 
where the last inequality is by Lemma \ref{philemma}.
\end{proof}

Since $S_l$ is generic for every $l$, the results of this section imply:
\begin{cor}
For any real number $r$, the set of selfmaps $f$ with
\[ N^\infty(f) \ge r, \quad h(f)\ge h_\#(f)\ge r, \]
is homotopy-generic.
\end{cor}

Jiang's paper \cite{jian96} also gives an upper bound on $N^\infty(f)$. Jiang's Propositions 2.1 and 2.6 show that $N^\infty(f)$ is bounded above by the greater of 1 or the spectral radius of the matrix $||\tilde F_1||$, where $\tilde F_1$ is the ``Fox Jacobian'' matrix (see \cite{fh83}) and the vertical bars denote taking the magnitude of group ring elements in each entry of the matrix. Directly from the definition of the Fox derivative we see that the $i,j$ entry of $||\tilde F_1||$ will equal $\Phi_{a_i} \phi(a_j)$.
We illustrate this with an example:
\begin{exl}
Let $m=2$ and let $f$ be a map so that $\phi = f_\#: \langle a, b \rangle \to \langle a, b \rangle$ is given by
\[ \phi(a) = \underline{ab^2a}b^{-1}, \quad \phi(b) = b\underline{ab^{-1}ab}. \]
The remnant above is underlined, and we can see that $\phi \in S_2$. Thus by Theorem \ref{ninfty} we have $N^\infty(f) \ge 2\cdot 2 = 4$.

For the upper bound, we compute $||\tilde F_1||$ and obtain:
\[ ||\tilde F_1|| = \begin{bmatrix} \Phi_a \phi(a) & \Phi_a \phi(b) \\ \Phi_b \phi(a) & \Phi_b \phi(b) \end{bmatrix} 
= \begin{bmatrix} 2 & 2 \\ 3 & 3 \end{bmatrix},
\]
which has spectral radius 5. Thus we have
\[ 4 \le N^\infty(f) \le 5. \]

Using a computer implementation of Wagner's algorithm (an implementation in the GAP language is available at the second author's website), we can compute values of $N(f^n)$. Though it is easily done by hand for ``small'' maps, Wagner's algorithm has exponential complexity, and so we are not able to compute many such values. In this example we can compute the following values (the third column is approximate):
\begin{center} \begin{tabular}{|c|c|c|}
\hline
$n$ & $N(f^n)$ & $N(f^n)^{1/n}$ \\
\hline
1 & 3 & 3 \\
2 & 19 & 4.358 \\
3 & 93 & 4.530 \\
4 & 431 & 4.556 \\
5 & 1973 & 4.560 \\
\hline
\end{tabular}
\end{center}
\end{exl}

\section{The density of $S_l$}\label{Dsection}
In this section we show that $D(S_l) =1$ for $l>0$. 

We will identify the set of endomorphisms $G \to G$ with the cartesian product $G^m$. Let $G_p$ be the set of elements of $G$ with length at most $p$. Then the density of $S_l$ will be 
\[ D(S_l) = \lim_{p \to \infty} \frac{|S_l \cap G_p^m|}{|G_p^m|}. \]

For $p>0$, the number of words of length exactly $p$ is 
\begin{equation}\label{wordsexactly}
\#\{w\in G : |w|=p\} = 2m(2m-1)^{p-1},
\end{equation}
since the first letter can be any letter of $G$, while each subsequent
letter can be anything but the inverse of the previous. Summing gives
the formula 
\begin{equation}\label{wordsatmost}
|G_p| = 1 + \sum_{j=1}^p 2m(2m-1)^{j-1} = \frac{m(2m-1)^p - 1}{m-1}
\end{equation}

The following basic fact about asymptotic density will be useful.
\begin{lem}\label{genericd}
For a free group $F$, let $A, B \subset F$ be subsets with $D(A) = 1$. Then if $D(A\cap B)$ exists, we have $D(A \cap B) = D(B)$.
\end{lem}
\begin{proof}
We compute:
\begin{align*}
D(B) - D(A \cap B) &= \lim_{p \to \infty} \frac{|B \cap F_p|}{|F_p|} - \frac{|A \cap B \cap F_p|}{|F_p|} = \lim_{p \to \infty} \frac{|B \cap F_p| - |A \cap B \cap F_p|}{|F_p|} \\
&\le \lim_{p \to \infty} \frac{|(B-A)\cap F_p|}{|F_p|} \le \lim_{p \to \infty} \frac{|(F-A) \cap F_p|}{|F_p|} \\ 
&= D(F-A) = 1-D(A) = 0. 
\end{align*}
\end{proof}

We are now ready for our final result:
\begin{thm}\label{densitythm}
For any $l>0$, we have $D(S_l) = 1$.
\end{thm}
\begin{proof}
We will prove the theorem in the special case where $l=1$. A straightforward but more complicated counting argument is required in the general case, which we will indicate below.
In the following, let $S=S_1$ and let $\bar S$ be the complement of $S$. 

For any $k>2$, we will define an equivalence relation among the homomorphisms of $R_k \cap G_p^m$ (for $p\ge k$): we say $\phi \sim \rho$ when for each $i$, the remnant words $\Rem_\phi a_i$ and $\Rem_\rho a_i$ have the same length and the same initial and terminal letters as each other. (So the initial letter of $\Rem_\phi a_i$ equals the initial letter of $\Rem_\rho a_i$ and the terminal letter of $\Rem_\phi a_i$ equals the terminal letter of $\Rem_\rho a_i$.)

The relation $\sim$ is an equivalence relation, let $[\phi] \subset R_k \cap G_p^m$ denote the equivalence class of $\phi$. We will estimate the quantities $\#[\phi]$ and $\#([\phi] \cap \bar S)$. Let $r_i = |\Rem_\phi a_i|$, and note that these numbers are the same among all maps in $[\phi]$. 

To construct a homomorphism $\rho \in [\phi]$, the initial and terminal letters of the remnant subword will be fixed. Thus we have choices only in filling in the subwords before and after the remnant, and the subword inside the remnant. Let $A_i$ and $B_i$ be the number of choices for the subwords before and after the remnant in $\rho(a_i)$, and let $R_i$ be the number of choices for the subword inside the remnant. The word inside the remnant can be any word of length $r_i - 2$ which does not cancel with the two chosen end letters. To choose this word we select the first letter to be different from the inverse of the previous, and continue choosing letters until the last one, which could possibly have the extra restriction of not being the inverse of the terminal letter of the remnant. Thus we have:
\[ (2m-1)^{r_i-3}(2m-2) \le R_i \le (2m-1)^{r_i-2}, \]
and so in particular we have
\begin{equation}\label{phisize}
\# [\phi] \ge \prod_{i=1}^m A_i B_i (2m-1)^{r_i-3}(2m-2). 
\end{equation}

Constructing a homomorphism $\psi \in [\phi] \cap \bar S$ is similar. Since $\psi \in \bar S$, there are some $j, k$ such that there is no $a_j^{\pm 1}$ inside some remnant subword $\Rem_\psi a_k$. By symmetry we will assume that $j=k=1$, and multiply our count by $m^2$ to account for other choices of $j$ and $k$. (This will multiply count the homomorphisms in which several remnant subwords lack certain letters.) 

The counts for $A_i$ and $B_i$ are the same for each $i$ as they were in \eqref{phisize}, as are the number of choices for the inside of the remnant in $\Rem_\psi a_i$ for $i\neq 1$. For $i=1$, however, $\Rem_\psi a_1$ cannot use the letters $a_1^{\pm 1}$ at all.\footnote{To prove the theorem for $S_l$ in the case with $l>1$, this count must be modified to allow no more than $l$ uses of $a_i^{\pm 1}$.} Thus the number of these choices will be less than or equal to $(2m-3)^{r_1 -2}$, and we have
\begin{equation}\label{phicapsize}
\# ([\phi] \cap \bar S) \le m^2 A_1 B_1 (2m-3)^{r_1-2}  \prod_{i=2}^m A_i B_i (2m-1)^{r_i-2} 
\end{equation}

Combining \eqref{phisize} and \eqref{phicapsize}, and recalling that $r_i \ge k$ for $\phi\in R_k$ gives
\begin{align*} 
\frac{\#([\phi] \cap \bar S)}{\#[\phi]} &\le \frac{m^2 (2m-3)^{r_1-2}}{(2m-1)^{r_1-3}(2m-2)}  = \frac{m^2}{(2m-2)(2m-1)}\left(\frac{2m-3}{2m-1}\right)^{r_1-2} \le \left( \frac{2m-3}{2m-1} \right)^{k-2}
\end{align*}
since $\frac{m^2}{(2m-2)(2m-1)} \le 1$ for $m\ge 2$.

Now we can measure the density $D(\bar S)$. We have:
\begin{align*}
D(\bar S) = D(\bar S \cap R_k) &= 
\lim_{p \to \infty} \frac{|\bar S \cap R_k \cap G_p^m|}{|G_p^m|} \\
&= \lim_{p \to \infty} \frac{|\bar S \cap R_k \cap G_p^m|}{|R_k\cap G_p^m|} \frac{|R_k \cap G_p^m|}{|G_p^m|} = \lim_{p \to \infty} \frac{|\bar S \cap R_k \cap G_p^m|}{|R_k\cap G_p^m|}
\end{align*}
since $R_k$ is generic. Let $\mathcal C$ be the set of equivalence classes in $R_k \cap G_p^m$ given by the relation $\sim$. Then splitting the above into classes gives:
\begin{align*}
D(\bar S) &= \lim_{p \to \infty} \frac{1}{|R_k\cap G_p^m|} \sum_{[\phi] \in \mathcal C} \#([\phi] \cap \bar S)
\le \lim_{p \to \infty} \frac{1}{|R_k\cap G_p^m|} \sum_{[\phi] \in \mathcal C} \#[\phi] \cdot \left( \frac{2m-3}{2m-1} \right)^{k-2} \\
&= \left( \frac{2m-3}{2m-1} \right)^{k-2} \lim_{p \to \infty} \frac{1}{|R_k\cap G_p^m|} \sum_{[\phi] \in \mathcal C} \#[\phi] \\
&= \left( \frac{2m-3}{2m-1} \right)^{k-2} \lim_{p \to\infty} \frac{1}{|R_k\cap G_p^m|} |R_k\cap G_p^m| 
= \left( \frac{2m-3}{2m-1} \right)^{k-2}
\end{align*}

Since the above is true for any $k$, and the final quantity can be made as small as desired with sufficiently large $k$, we have $D(\bar S) = 0$ and thus $D(S) = 1 - D(\bar S) = 1$.
\end{proof}

\end{document}